\documentclass[12pt]{article}
\usepackage{amsmath,amssymb,amsfonts,amsthm}
\newenvironment{myabstract}{\par\noindent
{\bf Abstract . } \small }
{\par\vskip8pt minus3pt\rm}
\newcounter{item}[section]
\newcounter{kirshr}
\newcounter{kirsha}
\newcounter{kirshb}
\newenvironment{enumarab}{\setcounter{kirshb}{1}
\begin{list}{(\arabic{kirshb})}{\usecounter{kirshb}} }{\end{list}}
\newtheorem{theorem}{Theorem}[section]

\newtheorem{corollary}[theorem]{Corollary}
\newenvironment{demo}[1]{\noindent{\bf #1.}\upshape\mdseries}
{\nopagebreak{\hfill\rule{2mm}{2mm}\nopagebreak}\par\normalfont}
\theoremstyle{definition}

\def\C{{\mathfrak{C}}}

\def\Nr{{\mathfrak{Nr}}}

\def\Sg{{\mathfrak{Sg}}}

\def\A{{\mathfrak{A}}}
\def\B{{\mathfrak{B}}}
\def\C{{\mathfrak{C}}}

\def\CA{{\bf CA}}

\def\Df{{\bf Df}}

\def\Lf{{\bf Lf}}

\def\RCA{{\bf RCA}}

\def\Rd{{\ Rd}}
\def\(R)RA{{\bf (R)RA}}

\def\c #1{{\cal #1}}
 \def\CA{{\sf CA}}
\def\B{{\sf B}}

\def\Nr{{\mathfrak{Nr}}}

\def\set#1{\{#1\} }

\def\Nr{{\mathfrak{Nr}}}

\def\A{{\mathfrak{A}}}
\def\B{{\mathfrak{B}}}
\def\C{{\mathfrak{C}}}

\def\A{{\mathfrak{A}}}
\def\B{{\mathfrak{B}}}
\def\C{{\mathfrak{C}}}

\def\Rd{{\mathfrak{Rd}}}

\def\CA{{\bf CA}}

\def\RCA{{\bf RCA}}

\def\c #1{{\cal #1}}

\def\Nr{{\mathfrak{Nr}}}

\def\CA{{\bf CA}}
\def\RCA{{\bf RCA}}
\def\c#1{{\mathcal #1}}

\def\set#1{ \{#1\}}

\def\Sg{{\mathfrak Sg}}

\def\Rl{{\mathfrak Rl}}

\title{For $2<n+1<m$, $Ur\Nr_n\CA_m$ is not finitely axiomatizable}
\author{Tarek Sayed Ahmed \\
Department of Mathematics, Faculty of Science,\\ 
Cairo University, Giza, Egypt.
  }
\begin{document}
\maketitle

\begin{myabstract} We show that for $2<n+1<m$, the class $\Nr_n\CA_m$ is psuedo elementary, whose
elementary that is not finitely axiomatizable.

\end{myabstract}



The class of neat reducts has been extensively studied by the author, 
Andr\'eka, N\'emeti, Hirsch, Hodkinson, Ferenzci and others. In this note we show that for $1<n<m$, the class 
$\Nr_n\CA_m $ is psuedo-elementary (it is known that it is not closed under ultraroots $Ur$ \cite{MLQ}), and that for
$2<n+1<m<\omega$, the class $Ur\Nr_nCA_m$ is not finitely axiomatizabe.
For our first result we use a defining theory in two sorts when both $n$ and $m$ are finite, three sorts 
when $n$ is finite and $m$ is infinite, and four sorts when both $m$and $n$ are infinite.
For our second result we use Monk-like algebras constructed 
by Robin Hirsch.

\begin{theorem}
Let $1<n<m$, then the class $\Nr_n\CA_m$ is pseudo-elementary, but is not elementary.
Furthermore, $EL\Nr_n\CA_m$is recursively enumerable, and for $n>2$, and $m\geq 2$, $EL\Nr_n\CA_m$ is not finitely axiomatizable.
\end{theorem}

\begin{demo}{Proof} 

\begin{enumarab}
\item For $n<m <\omega$, the charactersation is easy. One defines the class $\Nr_n\CA_m$ in a two sorted language. 
The first sort for the $n$ dimensional cylindric algebra the second for the $m$  dimensinal cylindric algebra. The signature of the defining theory 
includes an injective  function $I$ from sort one to sort two and includes a sentence requiring 
that $I$ respects the operations and a sentence saying that  an element of the second sort 
say $y$ satisfies  $\bigvee_{n\leq i<m} c_iy=y$, iff there exists $x$ of sort one 
such that $y=I(x)$ so that $I$ is a bijection.
 
Assume that $n$ is still finite, we first show that for any infinite $\alpha$, $\Nr_n\CA_{\omega}=\Nr_n\CA_{\alpha}$. Let $\A\in \Nr_n\CA_{\omega}$, 
so that $\A=\Nr_n\B'$, $\B'\in \CA_{\omega}$. Let $\B=\Sg^{\B'}A$. Then $\B\in \Lf_{\omega}$, and $\A=\Nr_n\B$. 
But $\Lf_{\omega}=\Nr_{\omega}\Lf_{\alpha}$ and we are done.
To show that $\Nr_n\CA_{\omega}\subseteq \Nr_n\RCA_{\omega}$, let $\A\in \Nr_n\CA_{\omega}$, then by the above argument
there exists  
then $\B\in \Lf_{\omega}$ such that $\A=\Nr_n\B$. by $\Lf_{\omega}\subseteq \RCA_{\omega},$ we are done. 

It is known that class $\Nr_n\CA_{\omega}$ is not elementary. In fact, there is an algebra $\A\in \Nr_n\CA_{\omega}$ 
having a complete subalgebra $\B$, and $\B\notin \Nr_n\CA_{n+1},$ this will be proved below.

Now assume that $m$ is infinite. Here if $y$ is in the $n$ dimensional 
cylindric algebra then we cannot express $c_i=y$ for all $i\in\omega\sim n$, like we did
when $m$ is finite, so we have to think differently.

To show that it is pseudo-elementary, we use a three sorted defining theory, with one sort for a cylindric algebra of dimension $n$ 
$(c)$, the second sort for the Boolean reduct of a cylindric algebra $(b)$
and the third sort for a set of dimensions $(\delta)$. We use superscripts $n,b,\delta$ for variables 
and functions to indicate that the variable, or the returned value of the function, 
is of the sort of the cylindric algebra of dimension $n$, the Boolean part of the cylindric algebra or the dimension set, respectively.
The signature includes dimension sort constants $i^{\delta}$ for each $i<\omega$ to represent the dimensions.
The defining theory for $\Nr_n\CA_{\omega}$ incudes sentences demanding that the consatnts $i^{\delta}$ for $i<\omega$ 
are distinct and that the last two sorts define
a cylindric algenra of dimension $\omega$. For example the sentence
$$\forall x^{\delta}, y^{\delta}, z^{\delta}(d^b(x^{\delta}, y^{\delta})=c^b(z^{\delta}, d^b(x^{\delta}, z^{\delta}). d^{b}(z^{\delta}, y^{\delta})))$$
represents the cylindric algebra axiom ${\sf d}_{ij}={\sf c}_k({\sf d}_{ik}.{\sf d}_{kj})$ for all $i,j,k<\omega$.
We have have a function $I^b$ from sort $c$ to sort $b$ and sentences requiring that $I^b$ be injective and to respect the $n$ dimensional 
cylindric operations as follows: for all $x^r$
$$I^b({\sf d}_{ij})=d^b(i^{\delta}, j^{\delta})$$
$$I^b({\sf c}_i x^r)= {\sf c}_i^b(I^b(x)).$$
Finally we require that $I^b$ maps onto the set of $n$ dimensional elements
$$\forall y^b((\forall z^{\delta}(z^{\delta}\neq 0^{\delta},\ldots (n-1)^{\delta}\rightarrow c^b(z^{\delta}, y^b)=y^b))\leftrightarrow \exists x^r(y^b=I^b(x^r))).$$

In this case we need a fourth sort. We leave the details to the reader.

In all cases, it is clear that any algebra of the right type is the first sort of a model of this theory. 
Conversely, a model for this theory will consist of an $n$ dimensional cylindric algebra type (sort c), 
and a cylindric algebra whose dimension is the cardinality of 
the $\delta$-sorted elements, which is at least $|m|$. Thus the first sort of this model must be a neat reduct.

\item For ${\A}\in \CA_n,$ $\Rd_3{\A}$ denotes the $\CA_3$
obtained from $\A$ by discarding all operations indexed by indices in $n\sim 3.$
$\Df_n$ denotes the class of diagonal free cylindric algebras.
$\Rd_{df}{\A}$ denotes the $\Df_n$ obtained from $\A$
by deleting all diagonal elements. To prove the non-finite axiomatizability result we use Monk's algebras given above.

For $3\leq n,i<\omega$, with $n-1\leq i, {\C}_{n,i}$ denotes
the $\CA_n$ associated with the cylindric atom structure as defined on p. 95 of 
\cite{HMT1}.
Then by \cite[3.2.79]{HMT1}
for $3\leq n$, and $j<\omega$, 
$\Rd_3{\C}_{n,n+j}$ can be neatly embedded in a 
$\CA_{3+j+1}$.  (1)
By \cite[3.2.84]{HMT1}) we have for every $j\in \omega$, 
there is an $3\leq n$ such that $\Rd_{df}\Rd_{3}{\cal C}_{n,n+j}$
is a non-representable $\Df_3.$ (2)
Now suppose $m\in \omega$. By (2), 
choose $j\in \omega\sim 3$ so that $\Rd_{df}\Rd_3{\C}_{j,j+m+n-4}$
is a non-representable $\Df_3$. By (1) we have
$\Rd_{df}\Rd_3{\C}_{j,j+m+n-4}\subseteq \Nr_3{\B_m}$, for some 
${\B}\in \CA_{n+m}.$
Put ${\A}_m=\Nr_n\B_m$.  
$\Rd_{df}{\A}_m$ is not representable, 
a friotri, ${\A}_m\notin \RCA_{n},$ for else its $\Df$ reduct would be 
representable. Therefore $\A_m\notin EL\Nr_n\CA_{\omega}$.
Now let $\C_m$ be an algebra similar to $\CA_{\omega}$'s such that $\B_m=\Rd_{n+m}\C_m$.
Then $\A_m=\Nr_n\C_m$. Let $F$ be a non-principal ultrafilter on $\omega$. Then
$$\prod_{m\in \omega}\A_m/F=\prod_{m\in \omega}(\Nr_n\C_m)/F=\Nr_n(\prod_{m\in \omega}\C_m/F)$$
But $\prod_{m\in \omega}\C_m/F\in \CA_{\omega}$. Hence $\CA_n\sim El\Nr_n\CA_{\omega}$ is not closed under ultraproducts.
It follows that the latter class is not finitely axiomatizable.
In \cite{IGPL} it is proved that for $1<\alpha<\beta$, $El\Nr_{\alpha}\CA_{\beta}\subset S\Nr_{\alpha}\CA_{\beta}$.

\item This follows from the folowing known fact (a result of Hirsch and Sayed ahmed, submitted for publication) 
For $3\leq m<n<\omega$, 
there is $m$ dimensional  algebra $\C(m,n,r)$ such that
\begin{enumarab}
\item $\C(m,n,r)\in \Nr_m\CA_n$
\item $\C(m,n,r)\notin S\Nr_m\CA_{n+1}$
\item $\prod_{r\in \omega} \C(m, n,r)\in El\Nr_m\CA_n$
\end{enumarab}

\end{enumarab}
\end{demo}

From the above proof it follows that
\begin{corollary} Let $K$ be any class such that $\Nr_n\CA_{\omega}\subseteq K\subseteq \RCA_n$. Then $ELK$ is not finitely axiomatizable
\end{corollary}

\begin{theorem} For $\alpha$ infinite, and $k\in \omega$ there is $\A\in \Nr_{\alpha}\CA_{\alpha+k}\sim S\Nr_{\alpha}\CA_{\alpha+k+1}$
\end{theorem}
\begin{proof} Let $\C(m,n,r)$ be as consrtucted above, then we also have:
For $m<n$ and $k\geq 1$, there exists $x_n\in \C(n,n+k,r)$ such that $\C(m,m+k,r)\cong \Rl_{x}\C(n, n+k, r).$
Let $\alpha$ be an infinite ordinal, 
let $X$ be any finite subset of $\alpha$, let $I=\set{\Gamma:X\subseteq\Gamma\subseteq\alpha,\; |\Gamma|<\omega}$.  
For each $\Gamma\in I$ let $M_\Gamma=\set{\Delta\in I:\Delta\supseteq\Gamma}$ and let $F$ be any ultrafilter over $I$ 
such that for all $\Gamma\in I$ we have $M_\Gamma\in F$ 
(such an ultrafilter exists because $M_{\Gamma_1}\cap M_{\Gamma_2} = M_{\Gamma_1\cup\Gamma_2}$).  
For each $\Gamma\in I$ let $\rho_\Gamma$ be a bijection from $|\Gamma|$ onto $\Gamma$.   
For each $\Gamma\in I$ let $\c A_\Gamma, \c B_\Gamma$ be $\CA_\alpha$-type algebras.  
For each $\Gamma\in I$ we have 
$\Rd^{\rho_\Gamma}\c A_\Gamma=\Rd^{\rho_\Gamma}\c B_\Gamma$ then $\Pi_{\Gamma/F}\c A_\Gamma=\Pi_{\Gamma/F}\c B_\Gamma$.  


Furthermore, if $\Rd^{\rho_\Gamma}\c A_\Gamma \in \CA_{|\Gamma|}$, for each $\Gamma\in I$  then $\Pi_{\Gamma/F}\c A_\Gamma\in \CA_\alpha$.

Let $k\in \omega$. Let $\alpha$ be an infinite ordinal. 
Then $S\Nr_{\alpha}\CA_{\alpha+k+1}\subset S\Nr_{\alpha}\CA_{\alpha+k}.$
Let $r\in \omega$. 
Let $I=\{\Gamma: \Gamma\subseteq \alpha,  |\Gamma|<\omega\}$. 
For each $\Gamma\in I$, let $M_{\Gamma}=\{\Delta\in I: \Gamma\subseteq \Delta\}$, 
and let $F$ be an ultrafilter on $I$ such that $\forall\Gamma\in I,\; M_{\Gamma}\in F$. 
For each $\Gamma\in I$, let $\rho_{\Gamma}$ 
be a one to one function from $|\Gamma|$ onto $\Gamma.$
Let ${\c C}_{\Gamma}^r$ be an algebra similar to $\CA_{\alpha}$ such that 
\[\Rd^{\rho_\Gamma}{\c C}_{\Gamma}^r={\c C}(|\Gamma|, |\Gamma|+k,r).\]
Let  
\[\B^r=\prod_{\Gamma/F\in I}\c C_{\Gamma}^r.\]
We will prove that 
\begin{enumerate}
\item\label{en:1} $\B^r\in \Nr_\alpha\CA_{\alpha+k}$ and 
\item\label{en:2} $\B^r\not\in S\Nr_\alpha\CA_{\alpha+k+1}$.  \end{enumerate}

The theorem will follow, since $\Rd_\CA\B^r\in S\Nr_\alpha \CA_{\alpha+k} \setminus S\Nr_\alpha\CA_{\alpha+k+1}$.

For the first part, for each $\Gamma\in I$ we know that $\c C(|\Gamma|+k, |\Gamma|+k, r) \in\CA_{|\Gamma|+k}$ and 
$\Nr_{|\Gamma|}\c C(|\Gamma|+k, |\Gamma|+k, r)\cong\c C(|\Gamma|, |\Gamma|+k, r)$.
Let $\sigma_{\Gamma}$ be a one to one function 
 $(|\Gamma|+k)\rightarrow(\alpha+k)$ such that $\rho_{\Gamma}\subseteq \sigma_{\Gamma}$
and $\sigma_{\Gamma}(|\Gamma|+i)=\alpha+i$ for every $i<k$. Let $\c A_{\Gamma}$ be an algebra similar to a 
$\CA_{\alpha+k}$ such that 
$\Rd^{\sigma_\Gamma}\c A_{\Gamma}=\c C(|\Gamma|+k, |\Gamma|+k, r)$.  By the second part   
with  $\alpha+k$ in place of $\alpha$,\/ $m\cup \set{\alpha+i:i<k}$ 
in place of $X$,\/ $\set{\Gamma\subseteq \alpha+k: |\Gamma|<\omega,\;  X\subseteq\Gamma}$ 
in place of $I$, and with $\sigma_\Gamma$ in place of $\rho_\Gamma$, we know that  $\Pi_{\Gamma/F}\A_{\Gamma}\in \CA_{\alpha+k}$.

We prove that $\B^r\subseteq \Nr_\alpha\Pi_{\Gamma/F}\c A_\Gamma$.  Recall that $\B^r=\Pi_{\Gamma/F}\c C^r_\Gamma$ and note 
that $C^r_{\Gamma}\subseteq A_{\Gamma}$ 
(the base of $C^r_\Gamma$ is $\c C(|\Gamma|, |\Gamma|+k, r)$, the base of $A_\Gamma$ is $\c C(|\Gamma|+k, |\Gamma|+k, r)$).
 So, for each $\Gamma\in I$,
\begin{align*}
\Rd^{\rho_{\Gamma}}\C_{\Gamma}^r&=\c C((|\Gamma|, |\Gamma|+k, r)\\
&\cong\Nr_{|\Gamma|}\c C(|\Gamma|+k, |\Gamma|+k, r)\\
&=\Nr_{|\Gamma|}\Rd^{\sigma_{\Gamma}}\A_{\Gamma}\\
&=\Rd^{\sigma_\Gamma}\Nr_\Gamma\A_\Gamma\\
&=\Rd^{\rho_\Gamma}\Nr_\Gamma\A_\Gamma
\end{align*}
By the first part of the first part we deduce that 
$\Pi_{\Gamma/F}\C^r_\Gamma\cong\Pi_{\Gamma/F}\Nr_\Gamma\A_\Gamma=\Nr_\alpha\Pi_{\Gamma/F}\A_\Gamma$,
proving \eqref{en:1}.

Now we prove \eqref{en:2}.
For this assume, seeking a contradiction, that $\B^r\in S\Nr_{\alpha}\CA_{\alpha+k+1}$, 
$\B^r\subseteq \Nr_{\alpha}\c C$, where  $\c C\in \CA_{\alpha+k+1}$.  
Let $3\leq m<\omega$ and  $\lambda:m+k+1\rightarrow \alpha +k+1$ be the function defined by $\lambda(i)=i$ for $i<m$ 
and $\lambda(m+i)=\alpha+i$ for $i<k+1$.
Then $\Rd^\lambda(\c C)\in \CA_{m+k+1}$ and $\Rd_m\B^r\subseteq \Nr_m\Rd^\lambda(\c C)$.
For each $\Gamma\in I$,\/  let $I_{|\Gamma|}$ be an isomorphism 
\[{\c C}(m,m+k,r)\cong \Rl_{x_{|\Gamma|}}\Rd_m {\c C}(|\Gamma|, |\Gamma+k|,r).\]
Let $x=(x_{|\Gamma|}:\Gamma)/F$ and let $\iota( b)=(I_{|\Gamma|}b: \Gamma)/F$ for  $b\in \c C(m,m+k,r)$. 
Then $\iota$ is an isomorphism from $\c C(m, m+k,r)$ into $\Rl_x\Rd_m\B^r$. 
Then $\Rl_x\Rd_{m}\B^r\in S\Nr_m\CA_{m+k+1}$. 
It follows that  $\c C (m,m+k,r)\in S\Nr_{m}\CA_{m+k+1}$ which is a contradiction and we are done.
\end{proof}

\end{document}